\documentclass[leqno,11pt, a4]{amsart}
\usepackage[utf8]{inputenc}
\usepackage{amsthm}
\usepackage{amsmath}
\usepackage{enumitem}
\usepackage{amsfonts}
\usepackage{amssymb}
\usepackage{hyperref}

\usepackage{graphicx}
\usepackage{pdfsync}
\usepackage{hyperref}
\usepackage{comment}
\usepackage{tikz}
\usepackage{hyperref}
\hypersetup{
    colorlinks=true,
    linkcolor=blue,
    filecolor=magenta,      
    urlcolor=blue,
    }

\usepackage{graphicx}
\usepackage{pdfsync}
\usepackage{xcolor}
\usepackage[backend=biber,
style=nature,
sorting=ynt
]{biblatex}
\usepackage{array}
\newtheorem{theorem}{Theorem}[section]
\newtheorem*{theorem*}{Theorem}
\newtheorem{definition}{Definition} [section]
\newtheorem{example}{Example}[section]
\newtheorem{corollary}{Corollary}[theorem]
\newtheorem{lemma}[theorem]{Lemma}
\newtheorem{proposition}[theorem]{Proposition}

\newcounter{yuppo}
\newcommand{\Keler} {K\"{a}hler }
\newcommand{\keler} {K\"{a}hler }

\renewcommand{\setminus}{-}
\newcommand{\om}{\omega}

\renewcommand{\phi}{\varphi}
\newcommand{\cinf}{C^\infty}

\newcommand{\ra}{\rightarrow}
\newcommand{\lra}{\longrightarrow}
\newcommand{\C}{\mathbb{C}}
\newcommand{\R}{\mathbb{R}}

\newcommand{\liu}{\mathfrak{u}}

\newcommand{\lia}{\mathfrak{a}}

\newcommand{\liek}{\mathfrak{k}}

\newcommand{\lieg}{\mathfrak{g}}

\newcommand{\liep}{\mathfrak{p}}

\newcommand{\la}{\lambda}

\newcommand{\sx}{\langle} % scalar product
\newcommand{\xs}{\rangle}
\newcommand{\liea}{\mathfrak{a}}
\newcommand{\noparty}[1]{}%{\colorbox{Apricot}{#1}}

\newcommand{\metrica}{(\, , \, )}

\newcommand{\mup}{\mu_\liep}

\title{Remarks on Semistable Points and Nonabelian Convexity of Gradient Maps}
\author{Oluwagbenga Joshua Windare}
\address{Dipartimento di Matematica e Informatica 'Ulisse Dini' (DIMAI) \\
         Universit\`a di Firenze (Italy)}
\email{oluwagbengajoshua.windare@unifi.it}
\thanks{The author was partially supported by GNSAGA of INdAM}

\begin{document}
	
\maketitle

\begin{abstract}
\noindent 
We study the action of a real reductive group $G$ on a \Keler manifold $Z$ which is the restriction of a holomorphic action of a complex reductive Lie group $U^\mathbb{C}.$ We assume that the action of $U$, a maximal compact connected subgroup of $U^\C$ on $Z$ is Hamiltonian. If  $G\subset U^\C$ is compatible, there is a corresponding gradient map $\mu_\mathfrak{p} : Z\to \mathfrak{p}$, where $\lieg = \liek \oplus \liep$ is a Cartan decomposition of the Lie algebra of $G$. Our main results are the openness and connectedness of the set of semistable points associated with $G$-action on $Z$, a convexity theorem for the $G$-action on a $G$-invariant compact Lagrangian submanifold of $Z$, and a convexity result for two-orbit variety.
\end{abstract}
	
	%\newpage 
	%\renewcommand{\contentsname}{Table of Contents}
	%\tableofcontents
	
\section{Introduction}

	\pagenumbering{arabic}
Let $U$ be a compact connected Lie group with Lie algebra $\mathfrak{u}$ and let $U^\C$ be its complexification. The Lie algebra $\mathfrak{u}^\C$ of $U^\C$ is the direct sum $\mathfrak{u}\oplus i\mathfrak{u}.$ A closed subgroup $G$ of $U^\C$ is compatible if $G$ is closed and the map $K\times \mathfrak{p} \to G,$ $(k,\beta) \mapsto k\text{exp}(\beta)$ is a diffeomorpism, where $K := G\cap U$ and $\mathfrak{p} := \mathfrak{g}\cap \text{i}\mathfrak{u};$ $\mathfrak{g}$ is the Lie algebra of $G.$ It follows that $G$ is compatible with the Cartan decomposition $U^\C = U\text{exp}(\text{i}\mathfrak{u})$, $K$ is a maximal compact subgroup of $G$ with Lie algebra $\mathfrak{k}$ and that $\mathfrak{g} = \mathfrak{k}\oplus \mathfrak{p}$ \cite{knapp-beyond}.

Let $(Z, \om)$ be a \Keler manifold. Let $U^\C$ acts
holomorphically on $Z$. Assume that $U$ preserves $\om$ and that there is a $U$-equivariant 
momentum map $\mu: Z \ra \liu^*$. Since $U$ is compact, we can identify $\liu$ with $\liu^*$ by an Ad$(U)$-invariant scalar product on $\liu$ denoted by $\sx\cdot, \cdot \xs$. Hence we may consider $\mu$ as a map $\mu: Z \ra \liu.$ If $\xi \in \liu$, we denote by $\xi_Z$
the induced vector field on $Z$ and we let $\mu^\xi \in \cinf(Z)$ be
the function $\mu^\xi(z) :=  \sx \mu(z),\xi\xs$. Then
$$
d\mu^\xi =
\iota_{\xi_Z} \om; \quad \text{grad}\mu^\xi = J\xi_Z.
$$ Let $\parallel\cdot\parallel$ denote a norm function. The function
$f : Z \rightarrow \mathbb{R}$ define by 
\begin{equation}\label{ns}
f(z) := \frac{1}{2}\parallel\mu(z)\parallel^2, \qquad \text{for}\quad z \in Z
\end{equation} is called the norm square of the momentum map.

Let $G\subset U^\C$ be a compatible subgroup. Then $\liep\subset i\liu.$ Let $\sx\cdot,\cdot\xs$ also denote the Ad$U$-invariant scalar product on $i\liu$ requiring multiplication by $i$ to be an isometry between $\liu$ and $i\liu$. If $z\in Z$, then the orthogonal projection of $i\mu(z)$ onto $\liep$ defines a $K$-equivariant map
\begin{gather*}
  \mu_\liep : Z \ra \liep
\end{gather*} called the \emph{gradient map}. $\mu_\liep^\beta(z) := \langle\mu_\liep(z), \beta\rangle = \langle i\mu(z), \beta\rangle = \langle \mu(z), -i\beta\rangle = \mu^{-i\beta}(z)$
% The map $\beta \mapsto -i\beta $ is an injection $\liep
% \hookrightarrow \liu$.
for any $\beta \in \mathfrak{p}$ \cite{PG,heinzner-schwarz-stoetzel,heinzner-schuetzdeller}. Let $\metrica$ be the \Keler
metric associated with $\om$, i.e. $(v, w) = \om (v, Jw)$ for all $z\in Z$ and $v,w\in T_zZ,$ where $J$ denotes the complex structure on $TZ$. Then $\beta_Z$ is the gradient of $\mup^\beta.$ Indeed,
$$\text{grad}\mu_\liep^\beta = \text{grad}\mu^{-i\beta} = J(-i\beta_Z) = \beta_Z.$$

Let $X$ be a $G$-invariant locally closed real submanifold of $Z.$ From now on, we denote the restriction of $\mu_\mathfrak{p}$ to $X$ by $\mu_\mathfrak{p}.$ Then
$$
\text{grad}\mu_\mathfrak{p}^\beta = \beta_X,
$$ where grad is computed with respect to the induced Riemannian metric on $X.$ 
%Similarly, $\bot$ denotes perpendicularity relative to the Riemannian metric on $X.$We will now recall some of the properties of the gradient map.

The norm square of the gradient map is the function $f_\liep : X \rightarrow \mathbb{R}$ defined by 
\begin{equation*}
\label{ns2}
f_\liep(x) := \frac{1}{2}\parallel\mu_\mathfrak{p}(x)\parallel^2, \qquad \text{for}\quad x \in X.
\end{equation*}

Let $X^{ss}_{\mu_\liep} := \{x\in X : \overline{G\cdot x}\cap \mu_\liep^{-1}(0) \neq\emptyset\}.$ The set $X^{ss}_{\mu_\liep}$ is called the set of semistable points. The set of semistable points has been studied in great detail in the literature \cite{Stability, heinzner-schwarz-stoetzel, heinzner-stoetzel, Teleman}. The set of semistable points is used to construct the topological Hilbert quotient for the $G$-action on $X$ (see Theorem \ref{good-quotient}). Our first main result is given below.
\begin{theorem}(Theorem \ref{Semistable-points})
Let $(Z, \omega)$ be a compact connected \keler manifold, and let $G\subset U^\C$ be a compatible subgroup. If $Z^{ss}_{\mu} \neq \emptyset$, then $Z^{ss}_{\mu_\liep}$ is an open dense connected subset of $Z.$
\end{theorem}

The second result is concerned with the convexity result of the gradient map. Suppose $X$ is compact and connected. The norm square $f_\liep$ defines a stratification on $X$ \cite{heinzner-schwarz-stoetzel}. In a situation where nonabelian convexity result holds, the minimal stratum, the stratum where the norm square attains it's minimum is unique. The uniqueness of the minimal stratum holds if $X = Z$ and $G = U^\C.$ In this situation, general convexity result is available \cite{gs, Kirwan2}. 

In a general setting, i.e, if $X\neq Z,$ there is still a minimal stratum and this stratum is always open. However, there could be more than one minimal strata and other open strata that are not minimal. We find a special case where the minimal stratum is the unique open stratum in $X.$

More precisely, we assume $G\subset U^\C$ is a real form. This means that $\lieg = \liek \oplus \liep,$ $\liu = \liek \oplus i\liep$ and $\lieg^\C = \liu^\C$. Then we prove the following result.

\begin{theorem}(Theorem \ref{unique stratum})
Suppose $X$ is a $G$-invariant compact connected Lagrangian submanifold of $Z$ such that $X\subset \mu_\liek^{-1}(0).$ Then there exists a unique open stratum in $X$. Moreover, this stratum is dense in $X.$
\end{theorem}

%\begin{remark}
%The condition that a $G$-invariant compact Lagrangian submanifold $X$ of $Z$ is such that $X\subset \mu_\liek^{-1}(0)$ is satisfied if $G$ is semisimple with finite fundamental group. 
%\end{remark}

Since $X$ is Lagrangian, $\mu_\liek$ restricted to $X$ is constant and the image lies in the center of the Lie algebra of $K.$ Hence, if $G$ has a finite fundamental group, then $K$ has a finite fundamental group and so $K$ is semisimple and the center of $K$ is finite. Summing up, the condition $X\subset \mu_\liek^{-1}(0)$ is always satisfied if $G$ has finite fundamental group.

Let $\liea$ be a maximal Abelian subalgebra of $\liep$ and $\liea_+$ a positive Weyl chamber of $\liea.$ Using Theorem \ref{unique stratum}, we prove the following nonabelian convexity theorem.

\begin{theorem}(Theorem \ref{convex})
Suppose $X$ is a $G$-invariant compact connected Lagrangian submanifold of $Z$ such that $X\subset \mu_\liek^{-1}(0).$ The set $\mu_\liep(X) \cap \liea_+$ is a convex polytope.
\end{theorem}

O'Shea and Sjamaar \cite{Shea} obtained a result similar to Theorem \ref{convex} by studying symplectic manifolds equipped with an anti-symplectic involution compatible with the momentum map. In this setting, if the fixed-point set of the involution is nonempty, it is a Lagrangian submanifold and the authors proved the Theorem for this fixed-point set. This is fairly natural from the point of view of \keler manifold. However, the hypothesis considered by the authors is slightly different from the settings considered in this article. The techniques in this article use the gradient map and the stratification theorem by Heinzner \cite{heinzner-schwarz-stoetzel}.

The third result is concerned with a convexity result of a two-orbit variety. Let $G\subset U^\C$ be a compatible subgroup. A compact connected $G$-invariant locally closed real submanifold of $Z$ is a two-orbit variety if $G$-action on $X$ has two orbits. In \cite{Properties}, the authors proved that the norm squared of the gradient map for a two-orbit variety is Bott-Morse. Using this result, we obtain the following result.

\begin{theorem*}(Theorem \ref{nab})
    If $X$ is a two-orbit variety, then the set $\mu_\liep(X) \cap \liea_+$ is a convex polytope.
\end{theorem*}

\section{Semistable Points}\label{analytic}

In this section, we study the set of semistable points of the gradient map using a result proved in \cite{Stability}. Let $(Z,\omega)$ be a \Keler manifold, and $U^\C$ acts holomorphically on $Z$ with a momentum map $\mu : Z \to \mathfrak{u}.$ Let $G\subset U^\C$ be a closed compatible subgroup. $G = K\exp(\mathfrak{p}),$ where $K := G\cap U$ is a maximal compact subgroup of $G$ and $\mathfrak{p} := \mathfrak{g}\cap \text{i}\mathfrak{u};$ $\mathfrak{g}$ is the Lie algebra of $G.$

Suppose $X\subset Z$ is a $G$-stable locally closed connected real submanifold of $Z$ with the gradient map $\mu_\mathfrak{p} : X\to \mathfrak{p}.$ The following properties of the Gradient map are well known.
\begin{lemma}\label{increasing}
Let $x\in X$ and let $\beta \in \mathfrak{p}.$ Then either $\beta_X(x) = 0$ or the function $t\mapsto \mu_\mathfrak{p}^\beta(\exp(t\beta) x)$ is strictly increasing.
\end{lemma}
\begin{corollary}\label{lem}
If $x \in X$ and $\beta \in \mathfrak{p},$ then $\mu_\mathfrak{p}(\exp(\beta) x) = \mu_\mathfrak{p}(x)$ if and only if $\beta_X(x) = 0.$
\end{corollary}

Let $G_x$ and $K_x$ denote the stabilizer subgroup of $x\in X$ with respect to the $G$-action and the $K$-action respectively and $\lieg_x$ and $\liek_x$ denote their respective Lie algebras.
\begin{definition}
Let $x\in X.$ Then:
\begin{enumerate}
\item $x$ is stable if $G\cdot x \cap \mu_\liep^{-1}(0) \neq \emptyset$ and $\lieg_x$ is conjugate to a Lie subalgebra of $\liek.$
\item $x$ is polystable if $G\cdot x \cap \mu_\liep^{-1}(0) \neq \emptyset.$
\item $x$ is semistable if $\overline{G\cdot x} \cap \mu_\liep^{-1}(0) \neq \emptyset.$
\end{enumerate}
\end{definition}
We denote by $X^s_{\mup}$, $X^{ss}_{\mup}$, $X^{ps}_{\mup}$ the set of stable, respectively semistable, polystable, points.
It follows directly from the definitions above that the conditions are $G$-invariant in the sense that if a point satisfies one of the conditions, then every point in its orbit satisfies the same condition, and for stability, recall that $\lieg_{gx} = \text{Ad}(g)(\lieg_x).$

One may define a relation $\sim$ on $X^{ss}_{\mup}$ where $x\sim y$ if $\mup^{-1}(0)\cap \overline{G\cdot x}\cap \overline{G\cdot y} \neq \emptyset$. This relation is indeed an equivalence relation \cite{PG} and we denote the corresponding quotient by $X^{ss}_{\mup}//G$ and call it the topological Hilbert quotient of $X^{ss}_{\mup}$ by the action of $G$. Let $\pi:X^{ss}_{\mup} \lra X^{ss}_{\mup}//G$ denote the quotient map. The results of Heinzner-Schwarz-St\"otzel \cite[Quotient Theorem p.164]{heinzner-schwarz-stoetzel} and \cite{PG,heinzner-stoetzel}, show the following theorem.
\begin{theorem}\label{good-quotient}
The subsets $X_{\mup}^s$, $X_{\mup}^{ss}$ are open subset in $X$. Moreover, the topological Hilbert quotient $X^{ss}_{\mup}//G$ has the following properties.
\begin{enumerate}
\item Every fiber contains a unique closed $G$-orbit. Any other orbit in the fiber has strictly larger  dimension.
\item The closure of every $G$-orbit in a fiber of $\pi$ contains the closed $G$-orbit.
\item every fiber of $\pi$ intersects $\mup^{-1}(0)$ in a unique $K$-orbit which lies in the unique closed $G$-orbit.
\item the inclusion $\mup^{-1}(0) \hookrightarrow X$ induces a homeomorphism $\mup^{-1}(0) /K \cong X^{ss}_{\mup} //G$.
\end{enumerate}
\end{theorem}
Therefore $X^{ss}_{\mup}//G$ can be identified with the space of polystable orbits. On the other hand the set of polystable points is in general neither open nor closed.
\subsubsection{Maximal Weight Function}
In this section, we introduce the numerical invariants $\la(x,\beta)$ associated to an element $x\in X$ and $\beta \in \mathfrak{p}.$

For any $t\in \mathbb{R},$ define $\la(x,\beta,t) = \langle\mu_\mathfrak{p}(\exp(t\beta)x), \beta\rangle.$
$$
\la(x,\beta,t) = \langle\mu_\mathfrak{p}(\exp(t\beta)x), \beta\rangle = \frac{d}{dt}\Phi(x, \exp(t\beta)),
$$ where $\Phi: X\times G\to \mathbb{R}$ is the Kempf-Ness function. By the properties of the Kempf-Ness function,
$$
\frac{d}{dt}\la(x,\beta,t) = \frac{d^2}{dt^2}\Phi(x, \exp(t\beta)) \geq 0.
$$
This means that $\la(x,\beta,t)$ is a non decreasing function as a function of $t.$

\begin{definition}
The maximal weight of $x\in X$ in the direction of $\beta \in \mathfrak{p}$ is the numerical value
$$
\la(x,\beta) = \lim_{t\to \infty}\la(x,\beta,t)\in \mathbb{R}\cup\{\infty\}.
$$
\end{definition}
From the proof of Lemma \ref{increasing} we have
$$\frac{d}{dt}\la(x,\beta,t) = \parallel\beta_X(\exp(t\beta) x)\parallel^2,
$$ and so
$$
\la(x,\beta,t) =  \langle\mu_\mathfrak{p}(x), \beta\rangle + \int_0^t\parallel\beta_X(\exp(s\beta) x)\parallel^2 \mathrm{ds}.
$$
For any $x\in X$ and $\beta \in \liep$, we consider the curve $c_x^\beta:[0,+\infty) \lra X$ defined by $c_x^\beta (t)=\exp(t\beta)x$. The energy functional of the curve  $c_x^\beta$ is given by
$$
E(c_x^\beta) = \int_0^{+\infty}\parallel \beta_X(\exp(t\beta)x) \parallel^2 \mathrm{dt}.
$$
Thus,
\begin{equation}\label{energy}
\la(x,\beta) = \la(x,\beta, 0) + E(c_x^\beta),
\end{equation}
\subsection{Energy Complete}
Let $\mup:X \lra \liep$ denote the $G$-gradient map associated with the momentum map $\mu:Z \lra \liu$.
\begin{definition}
The $G$-action on $X$ is called energy complete if
 for any $x\in X$ and for any $\beta \in \mathfrak{p}$, If $E(c_x^\beta ) < \infty$ then $\lim_{t\to \infty}\exp(t\beta)x$ exists.
\end{definition}
The proof of the following result is similar to \cite[Proposition 3.9]{Teleman}.
\begin{proposition}$\,$
\begin{enumerate}
\item If $X$ is compact, then the $G$-action is energy complete;
\item if $G$ acts on a complex vector space $(V,h)$, where $h$ is a $K$-invariant Hermitian scalar product, then the $G$-action is energy complete;
\item If $G\subset \mathrm{SL}(n,\R)$ is a closed and compatible, then the $G$-action on $\R^n$ is energy complete
\end{enumerate}
\end{proposition}
The reader can see some of the properties of energy complete action in \cite{Stability}.
\begin{definition}
A point $x\in X$ is called:
\begin{enumerate}
\item analytically stable if $\lambda(x,\beta) >0$ for any $\beta \in \liep \backslash\{0\}$;
\item analytically semi-stable if $\lambda(x,\beta) \geq 0$ for any $\beta \in \liep$;
\item analytically polystable if $\lambda(x,\beta) \geq 0$ for any $\beta \in \liep$ and the condition $\lambda(x,\beta)=0$ holds if and only if $\lim_{t\mapsto +\infty} \exp(t\beta) \in G\cdot x$.
\end{enumerate}
\end{definition}
The author together with Biliotti gave a numerical criterion for semistable points. This criterion is given below.
\begin{theorem}\cite{Stability}\label{semistable}
 Let $x\in X$. The following conditions are equivalent:
\begin{itemize}
\item[(1)]  $x$ is semistable.
\item[(2)]  $\mathrm{Inf}_G \parallel \mu(gx) \parallel=0$.
\item[(3)]  $x$ is analytically semistable.
\end{itemize}
\end{theorem}

\begin{corollary}\label{hmss}
Let  $x\in X$. Then $x$ is semistable if and only if there exists $\xi \in \liep$ and $g\in (G^{\xi})^o$ such that $\lim_{t\mapsto +\infty} \exp(t\xi) gx \in \mup^{-1}(0)$.
\end{corollary}

%\section{Semistable points}\label{Semistability of subgroup}
Let $\lia \subset \liep$ be a maximal Abelian subalgebra and let  $A = \exp(\liea)$. Then $\mu_\liea:X \lra \liea$ given by $\pi_\lia \circ \mu_{\liep}$ is the $A$-gradient map associated to $\mu$, where $\pi_\lia:\liep \lra \lia$ is the orthogonal projection of $\liep$ onto $\lia$. Let
$$
X^{ss}_{\mu_\lia} := \{x\in X : \overline{A\cdot x}\cap \mu_\liea^{-1}(0) \neq\emptyset\}.
$$ Let $\lia' \subset \liep$ be another maximal Abelian subalgebra. Since the $K$-action on $\liep$ is polar and both $\lia$ and $\lia'$ are section then there exists $k\in K$ such that $\mathrm{Ad}(k)(\liea) = \liea'$ \cite{knapp-beyond}. In particular,
$$
 A' = \exp(\liea') = \exp(\mathrm{Ad}(k)(\liea)) = k\exp(\liea)k^{-1} = kAk^{-1}.
$$

\begin{lemma}\label{Abelian subalgebra} The $A'$-gradient map $\mu_{\liea'}:X \lra \liea'$ associated to $\mu$ is such that
$$\mu_{\liea'} = \mathrm{Ad}(k) \circ \mu_\liea \circ k^{-1},\quad k\in K$$
\end{lemma}
\begin{proof}
Let $\xi'\in \liea'.$ Then, there exist $k\in K$ such that $\xi' = \mathrm{Ad}(k)(\xi)$ where $\xi \in \liea.$ So that by the $K$-equivariant property of the gradient map,
\begin{align*}
    \sx \mu_\liep(x), \xi'\xs & = \sx \mu_\liep(k^{-1}x), \xi\xs\\
    & = \sx \mathrm{Ad}(k)(\mu_\liep(k^{-1}x)), \xi'\xs.
\end{align*} The result follows.

\end{proof}

\begin{lemma}\label{Abelian Sub} There exists $k\in K$ such that
$$kX^{ss}_{\mu_\liea} = X^{ss}_{\mu_{\liea'}}.$$
\end{lemma}
\begin{proof}
Since $A' = kAk^{-1},$ and by Lemma \ref{Abelian subalgebra}, $\mu_{\liea'}(x) = 0$ if and only if $\mu_{\liea}(k^{-1}x) = 0$. This implies that $\mu_{\liea'}^{-1}(0) = k\mu_{\liea}^{-1}(0).$ The sequence $(\exp(\xi_n) x)_n$ converges in $\mu_{\liea}^{-1}(0)$ if and only if $(k\exp(\xi_n) x)_n$ converges in $\mu_{\liea'}^{-1}(0).$ But
$$ k\exp(\xi_n) x = k\exp(\xi_n)k^{-1}k x = \exp(\mathrm{Ad}(k)(\xi_n))k x.$$
This implies that $\overline{A\cdot x}\cap \mu_\liea^{-1}(0) \neq \emptyset$ if and only if $\overline{A'\cdot (k x)}\cap \mu_{\liea'}^{-1}(0) \neq \emptyset.$ Hence $kX^{ss}_{\mu_\liea} = X^{ss}_{\mu_{\liea'}}.$
\end{proof}
\begin{proposition}\label{Semi}
$X^{ss}_{\mu_\liep} =\bigcap_{k\in K} kX^{ss}_{\mu_{\liea}}.$
\end{proposition}
\begin{proof}
By Theorem \ref{semistable}, $x\in X^{ss}_{\mu_\liep}$ if and only if $\lambda(x,\beta) \geq 0.$ Since the $K$-action on $\liep$ is polar, by \cite[13.1]{ PG}
$$
\liep = \bigcup_{k\in K}\mathrm{Ad}(k)(\liea).
$$ Then by Lemma \ref{Abelian Sub}, we have
\begin{align*}
    x\in X^{ss}_{\mu_\liep} &\iff \lambda(x, \cdot)|_{\mathrm{Ad}(k)(\liea)} \geq 0 \quad \forall k\in K\\
    &\iff x\in X^{ss}_{\mu_{\liea'}},\quad  \liea' = \mathrm{Ad}(k)(\liea) \; \forall k\in K\\
    &\iff kx\in X^{ss}_{\mu_{\liea}} \quad \forall k\in K.
\end{align*} Therefore, $X^{ss}_{\mu_\liep} =\bigcap_{k\in K} kX^{ss}_{\mu_{\liea}}.$
\end{proof}
\begin{proposition}\label{semistable set}
Let $H\subset G$ be a compatible subgroup such that $H = K'\exp(\liep').$  Then $$X^{ss}_{\mu_\liep} \subset X^{ss}_{\mu_{\liep'}}.$$
\end{proposition}
\begin{proof}
 Let $x\in X^{ss}_{\mu_\liep}$. By Theorem \ref{semistable} it follows that $\lambda(x,\beta)\geq 0$ for any $\beta\in \liep$. Hence $\lambda(x,\beta)\geq 0$ for any $\beta \in \liep'\subset \liep$ and so, applying again Theorem \ref{semistable}, we have $x\in X^{ss}_{\mu_{\liep'}}$.
\end{proof}
\begin{theorem}\label{Semistable-points}
Let $(Z, \omega)$ be a compact connected \keler manifold, and let $G\subset U^\C$ be a compatible subgroup. If $Z^{ss}_{\mu} \neq \emptyset$, then $Z^{ss}_{\mu_\liep}$ is an open dense connected subset of $Z.$
\end{theorem}

\begin{proof}
By Proposition \ref{semistable set}, $Z^{ss}_\mu\subset Z^{ss}_{\mu_\liep}.$ Since $Z^{ss}_\mu$ is open and dense in $Z,$ it follows that $Z^{ss}_{\mu_\liep}$ is open and dense. Suppose $$Z^{ss}_{\mu_\liep} = \Upsilon_1\sqcup \Upsilon_2,$$ where $\Upsilon_1$ and $\Upsilon_2$ are open. Then $Z^{ss}_\mu \subset \Upsilon_1$ or $Z^{ss}_\mu \subset \Upsilon_2.$ If $Z^{ss}_\mu \subset \Upsilon_1$, this implies that $\Upsilon_2 \subset Z\setminus Z_\mu^{ss}.$ Since $Z^{ss}_\mu$ is open and dense, then $(Z\setminus Z_{\mu}^{ss})^0 = \emptyset.$ This implies that $\Upsilon_2 = \emptyset.$ The case $Z_\mu^{ss}\subset \Upsilon_2$ is similar. Therefore, $Z_{\mu_\liep}^{ss}$ is connected, open and dense.

\end{proof}

\section{Convexity Results}

\subsection{Norm square of Gradient map}
\label{subsection-gradient-moment}
Suppose $(Z, \omega)$ is a compact connected \keler manifold and that the action $U^\C \times Z \to Z$ is holomorphic, $U$ preserves $\om$ and that there is a $U$-equivariant 
momentum map $\mu: Z \ra \liu.$ Let $G \subset U^\C$ be a compatible subgroup of $U^\C.$ For a $G$-invariant locally closed real submanifold $X$ of $Z,$ let $\mu_\mathfrak{p} : X\to \mathfrak{p}$ be the corresponding gradient map, where $\lieg = \liek \oplus \liep$ is a Cartan decomposition of the Lie algebra $\lieg$ of $G.$

The norm square of the gradient map is the function $f_\liep : X \rightarrow \mathbb{R}$ defined by 
\begin{equation}\label{ns3}
f_\liep(x) := \frac{1}{2}\parallel\mu_\mathfrak{p}(x)\parallel^2, \qquad \text{for}\quad x \in X.
\end{equation}
%\begin{lemma}\label{Equivalent}
%Let $x\in X$ such that $\mu_\mathfrak{p}(x) = 0.$ The following are equivalent:

%\begin{enumerate}
%    \item $d\mu_\mathfrak{p}(x) : T_xX \to \mathfrak{p}$ is onto
%    \item $\tau_1 : \mathfrak{p} \to T_xX;$ $\beta \mapsto \beta_X(x)$ is injective.
%    \item $\tau_2 : \mathfrak{g} \to T_xX;$ $\xi \mapsto \xi_X(x)$ is injective.
%\end{enumerate}
%\end{lemma}

%\begin{proof}
%Let $\beta\in \text{ker}\;\tau_1.$ For all $v\in T_xX,$ $ \langle \beta_X(x), v\rangle = 0$ $\Longleftrightarrow \langle d\mu_\mathfrak{p}(x)(v), \beta \rangle = 0$ and this shows that $(a)$ is equivalent to $(b).$

%$(b)$ $\implies$ $(c)$: Let $\xi\in \text{ker}\;\tau_2.$ Since $\mathfrak{g} = \mathfrak{k} \oplus \mathfrak{p},$ we can set $\xi = \alpha + \beta,$ where $\alpha\in \mathfrak{k}$ and $\beta\in \mathfrak{p}.$ Then
%\begin{align*}
%    0 &= d\mu_\mathfrak{p}(x)(\alpha_X(x) + \beta_X(x))\\
%    & = d\mu_\mathfrak{p}(\alpha_X(x)) - [\mu_\mathfrak{p}(x), \beta_X(x)]\\
%    & = d\mu_\mathfrak{p}(\alpha_X(x))
%\end{align*}
%Hence, $\alpha_X(x) = 0.$ But $0 = \xi_X(x) = \alpha_X(x) + \beta_X(x).$ Which implies that $\beta_X(x) = 0$ and hence $\alpha = \beta = 0.$ $(c)$ $\implies$ $(b)$ is obvious concluding the proof.

%\end{proof}

For the rest of this paper, we assume $X$ is compact and connected. We recall some properties of the norm square of the Gradient map.

\begin{lemma}\cite{heinzner-schwarz-stoetzel}\label{nmg}
The gradient of $f_\liep$ is given by 
\begin{equation}\label{nmge}
    \triangledown f_\liep(x) = \beta_X(x), \quad \beta := \mu_\mathfrak{p}(x)\in \mathfrak{p} \quad \text{and}\quad x\in X.
\end{equation} Hence, $x\in X$ is a critical point of $f_\liep$ if and only if $\beta_X(x) = 0.$
\end{lemma}

%\begin{proof}

%Define a curve $\gamma(t)$ such that $\gamma(0) = x$ and $\gamma'(0) = v \in T_xX.$ $f(x) = \frac{1}{2}|\mu_\mathfrak{p}(x)|^2 = \frac{1}{2}\langle \mu_\mathfrak{p}(x), \mu_\mathfrak{p}(x)\rangle$ 
%\begin{align*}
 %   df(x)v &= \frac{d}{dt}|_{t=0}f(\gamma(t))\\
 %   & = \frac{1}{2}\frac{d}{dt}|_{t=0}\langle \mu_\mathfrak{p}(\gamma(t)), \mu_\mathfrak{p}(\gamma(t))\rangle \\
 %   & = \langle d\mu_\mathfrak{p}(\gamma(t))\gamma'(t), \mu_\mathfrak{p}(\gamma(t))\rangle|_{t=0}\\
  %  & = \langle d\mu_\mathfrak{p}(x)v, \mu_\mathfrak{p}(x)\rangle = \langle \beta_X(x), v), \quad \beta = \mu_\mathfrak{p}(x)
%\end{align*}
%Hence, $\triangledown f(x) = \beta_X(x)$
%\end{proof}

\begin{corollary}\cite{heinzner-schwarz-stoetzel}
  Let $x\in X$ and set $\beta := \mu_\mathfrak{p}(x).$ The following are equivalent.
 \begin{enumerate}
      \item $\beta_X(x) = 0,$
      \item $d\mu_\mathfrak{p}^\xi(x) = 0,$ $\xi \in \mathfrak{p},$
      \item $df_\liep(x) = 0.$
  \end{enumerate}
\end{corollary}

Fix $\beta = \mu_\mathfrak{p}(x).$ The negative gradient flow line of $f_\liep$ through $x\in X$ is the solution of the differential equation

\[ \left\{ \begin{array}{ll}
         \dot{x}(t) = -\beta_X(x(t)), \quad t\in \mathbb{R} \\
        x(0) = x .\end{array} \right. \]
        
The $G$-orbits are invariant under the gradient flow.
\begin{lemma}\cite{Properties}\label{Gradient}
Let $g: \mathbb{R} \rightarrow G$ be the unique solution of the differential equation
\[ \left\{ \begin{array}{ll}
         g^{-1}\dot{g}(t) = \beta_X(x(t)) \\
        g(0) = e,\quad \text{where $e$ is the identity of $G$}.\end{array} \right. \]
        
Then, $$x(t) = g^{-1}(t)x$$ for all $t\in \mathbb{R}.$
\end{lemma}

%\begin{proof}
%Define $y : \mathbb{R} \to X$ by $$y(t) = g^{-1}(t)x.$$ Since $\dot{g^{-1}} = -g^{-1}\dot{g}g^{-1}$ and $g^{-1}\dot{g} = \beta_X(x),$ it follows that 
%$$\dot{y} = -g^{-1}\dot{g}g^{-1}x =  -\beta_X(g^{-1}x) = -\beta_X(y(t))$$ and
%$$y(0) = (g(0))^{-1}x = e^{-1}x = x$$

%Hence $x(t) = y(t) = g^{-1}(t)x$ for all $t\in \mathbb{R}.$

%\end{proof}

\begin{theorem}\cite{Properties}\label{teo}
Let $x_0\in X$ and $x: \mathbb{R} \rightarrow X$ the negative gradient flow line of $f$ through $x_0$. There exist positive constants $\alpha,$ $C,$ $\psi,$ and $\frac{1}{2} < \gamma < 1$ such that

$$x_\infty := \lim_{t \rightarrow \infty} x(t) $$
exists. Moreover, there exist a constant $T > 0$ such that for any $t > T,$ 
\begin{align*}
    d(x(t), x_\infty) & \leq \int_t^\infty |\dot{x}(s)|ds\\
    &\leq \frac{\alpha}{1 - \gamma}(f(x(t)) - f(x_\infty))^{1-\gamma}\\
    & \leq \frac{C}{(t - T)^\psi}.
\end{align*}
\end{theorem}

\subsection{Stratifications of the Norm Square of Gradient map.}\label{ssss}
We recall the stratification theorem. For details see \cite{heinzner-schwarz-stoetzel}. Let $C$ denote the critical set of $f_p$, $\mathfrak{B} := \mu_\mathfrak{p}(C)$ and $\mathfrak{B}_+ := \mathfrak{B}\cap \mathfrak{a}_+.$
Let $X^{ss} := \{x\in X : \overline{G\cdot x} \cap \mu_\mathfrak{p}^{-1}(0) \neq \emptyset \},$ For $\beta \in \mathfrak{B}_+,$ set
\begin{align*}
    & X|_{|\beta|^2} := \{x\in X : \overline{\text{exp}(\mathbb{R}\beta)\cdot x} \cap (\mu_\mathfrak{p}^\beta)^{-1}(|\beta|^2) \neq \emptyset \}\\
    & X^\beta := \{x\in X : \beta_X(x) = 0 \}\\
   & X^\beta|_{|\beta|^2} := X^\beta \cap X|_{|\beta|^2}\\
   & X^{\beta +}|_{|\beta|^2} := \{x\in X|_{|\beta|^2}: \lim_{t\to -\infty}\text{exp}(t\beta)\cdot x \; \text{exists and it lies in} \; X^\beta|_{|\beta|^2}\}
\end{align*}
 For $\beta\in \mathfrak{p},$ let $$G^{\beta+} :=\{g \in G : \lim_{t\to - \infty} \text{exp} ({t\beta}) g
  \text{exp} ({-t\beta}) \text { exists} \}.$$
  $G^{\beta+}$ is a parabolic subgroup with Levi factor $$G^\beta=\{g\in G:\, \mathrm{Ad}(g)(\beta)=\beta\}.$$
  
  $G^\beta$ is a compatible subgroup of $(U^\beta)^\C = (U^\C)^\beta.$

The set $X^{\beta +}|_{|\beta|^2}$ is $G^{\beta +}$-invariant. Let $\mu_{\mathfrak{p}^\beta}$ be a gradient map of the $G^\beta$-action on $X^{\beta +}|_{|\beta|^2}.$ Since $\beta$ is in the center of $\mathfrak{g}^\beta$, we have a shifted gradient map $$\widehat{\mu_{\mathfrak{p}^\beta}} := \mu_{\mathfrak{p}^\beta} - \beta.$$ Let 
$$
S^{\beta +} := \{x \in X^{\beta +}|_{|\beta|^2}: \overline{G^\beta \cdot x} \cap \mu_{\mathfrak{p}^\beta}^{-1}(\beta) \neq \emptyset \}.
$$ $S^{\beta +}$ coincides with the set of semistable points of $G^\beta$ in $X^{\beta +}|_{|\beta|^2}$ after shifting.

\begin{definition}
The $\beta$-stratum of $X$ is given by $S_\beta := G\cdot S^{\beta +}.$
\end{definition}

\begin{theorem}\label{Stratification}
(Stratification Theorem \cite[7.3]{heinzner-schwarz-stoetzel}). Suppose $X$ is a compact $G$-invariant submanifold of $Z.$ Then $\mathfrak{B}_+$ is finite and 
$$
X = \bigsqcup_{\beta \in \mathfrak{B}_+}S_\beta.
$$
Moreover, each $S_\beta$ is a locally closed submanifold of $X$ and
$$
\overline{S_\beta}\subset S_\beta \cup \bigcup_{|\gamma| > |\beta|}S_\gamma.
$$
\end{theorem}

\begin{proposition}\cite[6.12]{heinzner-schwarz-stoetzel} \label{lemm}
If $z\in X$ satisfies $$f_\liep(z) = max_{x\in X}f_\liep(x).$$ Then $G\cdot z = K\cdot z$ and so it is closed orbit.
\end{proposition}

Suppose $G\subset U^\C$ is a real form. This means that $\lieg = \liek \oplus \liep$ implies $\liu = \liek \oplus i\liep$ and the momentum map $\mu$ decompose to $\mu = \mu_\liek \oplus -i\mu_\liep.$ We fix a $G$-invariant compact Lagrangian submanifold $X$ of $Z$ such that $X\subset \mu_\liek^{-1}(0).$

\begin{lemma}\label{critical points}
$x\in X$ is a critical point of $f_\liep$ if and only if $x$ is a critical point of $f.$
\end{lemma}
\begin{proof}
Since $\mu$ decompose to $\mu = \mu_\liek \oplus -i\mu_\liep$ and $X\subset \mu_\liek^{-1}(0)$ then, $\mu|_X = -i\mu_\liep|_X.$ This implies that for $x\in X,$ the negative gradient flow of $f_p$ through $x$ coincides with the gradient flow of $f$ through $x$ and so, the result follows.
\end{proof}
Lemma \ref{critical points} asserts that for $\beta \in \liep,$ $K\cdot\beta$ is a critical orbit of $f_\liep$ if and only if $U\cdot (-i\beta)$ is a critical orbit of $f.$

\begin{theorem}\label{unique stratum}
Suppose $X$ is a $G$-invariant compact connected Lagrangian submanifold of $Z$ such that $X\subset \mu_\liek^{-1}(0).$ Then there exists a unique open stratum in $X$. Moreover, this stratum is dense in $X.$
\end{theorem}

\begin{proof}

It is well known that the strata associated with $f$ are locally closed submanifolds of $Z$ of even dimension \cite{Kirwan}. Since it is impossible to disconnect a manifold by removing submanifolds of codimension at least two, there must be a unique open stratum. Let $\overline{S_{-i\beta}}$ be the stratum associated with the minimum of $F$ for the $U^\C$-action on $Z.$ $\overline{S_{-i\beta}}$ is the unique open stratum of $f$. We set $S_\beta = \overline{S_{-i\beta}}\cap X$ which is a stratum of $f_\liep$ by Lemma \ref{critical points}. We claim that $S_\beta$ is the minimal stratum of $f_\liep$. Indeed, For $p\in S_\beta,$ $T_pS_\beta = T_pX \subset T_p\overline{S_{-i\beta}}$ and $J(T_pX)\subset T_p\overline{S_{-i\beta}}.$ This means that $T_pS_\beta \oplus J(T_pS_{\beta}) = T_pZ.$ This implies that $S_\beta$ is open.
Suppose it is not the unique open stratum. Then for each open stratum $S_{\widetilde{\beta}}$, there exists $\overline{S_{-\widetilde{i\beta}}}$ such that $S_{\widetilde{\beta}} = \overline{S_{-\widetilde{i\beta}}}\cap X.$ Then $\overline{S_{-\widetilde{i\beta}}}$ is open. But $\overline{S_{-i\beta}}$ is unique. It follows that $S_\beta$ is the unique open strata associated with $f_\liep.$

Since $X$ is compact, there are finitely many strata \cite[8.5]{heinzner-schwarz-stoetzel}. But $X$ has exactly one open stratum and the other strata are codimension at least one. This implies that the union of these other strata has empty interior. Thus, the open stratum $S_{\beta}$ is dense.
\end{proof}

Let $\beta\in \liep.$ The orbit $U\cdot \beta$ can be identified with the coadjoint orbit $U\cdot (-i\beta)\subset \liu$ which is a complex flag manifold. There is an induced $U^\C$-action on $U\cdot \beta$ and the map $\mu_\beta : Z\times U\cdot \beta \to \liu;$ $(x,\xi)\mapsto \mu(z)-Ad_g\xi$ with $g\in U$ is the moment map for the $U$-action on $Z\times U\cdot \beta.$ There is a corresponding Gradient map called the shifting of $\mu_\liep$ with respect to $\beta$ given as $\mu_{\liep, \beta} : X \times K\cdot \beta \to \liep, (x,\xi) \mapsto \mu_\liep - Ad_k\beta.$

By a result in \cite{Gorodski}, $K\cdot \beta$ is a Lagrangian submanifold of $U\cdot \beta.$ Hence, $X\times K\cdot \beta$ is a Lagrangian submanifold of $Z\times U\cdot \beta$ and $X\times K\cdot \beta \subset \mu_{\liek, \beta}^{-1}(0).$

\begin{corollary}
In the above assumption, the norm square of the gradient map $\mu_{\liep, \beta}$ has a unique open stratum that is dense.
\end{corollary}

\begin{theorem}\label{convex}
Suppose $X$ is a $G$-invariant compact connected Lagrangian submanifold of $Z$ such that $X\subset \mu_\liek^{-1}(0).$ The set $\mu_\liep(X) \cap \liea_+$ is a convex polytope.
\end{theorem}

\begin{proof}
The proof follows the geometric idea used in Kirwan's proof of convexity result in \cite{Kirwan2}. The proof used the fact that if a subset $D$ of a Euclidean vector space is a finite union of convex polytopes and is not convex, then for any sufficiently small $r > 0$ there exists a point $\beta\in \liea$ such that the closed ball centered at $\beta$ with radius $r$ meets $D$ in precisely two points.

$\mu_\liep(X) \cap \liea_+$ is a finite union of convex polytopes \cite[8.5]{heinzner-schuetzdeller}.
Suppose $\mu_\liep(X) \cap \liea_+$ is not convex. There is a point $\beta\in \liea_+$ and $r > 0,$ such that the boundary of the closed ball $B_r(\beta)$ meets $\mu_\liep(X) \cap \liea_+$ in at least two points. This implies that the function $\parallel \mu_{\liep, \beta} \parallel^2$ attains it's minimum value in two or more points. Suppose $\xi_1$ and $\xi_2$ are two of the minimum points of $\parallel \mu_{\liep, \beta}\parallel^2,$ then since $\xi_1 \neq \xi_2$ and keeping in mind that $K\cdot \xi_1 \cap \mathfrak a_{+}=\{\xi_1\}$, by the stratification theorem \cite{heinzner-schwarz-stoetzel}, see also Theorem \ref{Stratification},  we have two disjoint open strata $S_{\xi_1}$ and $S_{\xi_2}$ and this is a contradiction of the above lemma.
\end{proof}

\begin{example}
Let $Z = Gr(4, \C^8),$ a Grassmannian 4-planes. Let $B\in SU(4)$ such that $\overline{B} \neq B.$ Set 
\begin{equation*}
     A =
\left(\begin{tabular}{ m{0.5cm}| m{0.5cm} } 
  Id & 0\\ 
  \hline
 0 & B\\ 
\end{tabular}\right)
\end{equation*}

Consider the map $\Phi_A : Gr(4, \C^8) \to Gr(4, \C^8);$ $\pi \mapsto A(\pi).$ $\Phi_A$ is an holomorphic isometries and it preserves $\omega.$ Hence, $G_{\Phi_A} = \{(\pi, \Phi_A(\pi)): \pi \in Gr(4, \C^8) \}$ is Lagrangian in $(Gr(4, \C^8) \times Gr(4, \C^8), (\omega \times -\omega)).$ Choose $SL(4, \C) \subset GL(8, \C);$ 
\begin{equation*}
     D  \mapsto
\left(\begin{tabular}{ m{0.5cm}| m{0.5cm} } 
  D & 0\\ 
  \hline
 0 & Id\\ 
\end{tabular}\right).
\end{equation*}

 $SL(4, \C)$ acts on $Gr(4, \C^8)\times Gr(4, \C^8)$ diagonally; $D(\pi_1, \pi_2) = (D(\pi(\pi_1), D(\pi_2)).$
 $SL(4, \C)$ preserves $G_{\Phi_A}.$ Indeed, $SL(4, \C) \subset GL(8, \C)$ centralizes $A.$ Therefore, for $g\in SL(4,\C),$
 \begin{align*}
     g(\pi, A(\pi)) &= (g(\pi), g(A(\pi)))\\
     & = (g(\pi), A(g(\pi)).
 \end{align*}

$SL(4, \R)$ is a real form of $SL(4, \R).$ Theorem \ref{convex} applies to $SL(4,\R)$-action on $G_{\Phi_A}.$

\end{example}

\subsection{Two-orbit Variety}

This section proves the non-Abelian convexity results for the two-orbit variety.
\begin{definition}
Let $X$ be a compact and connected $G$-stable submanifold of $(Z,\omega)$. $X$ is a two-orbit variety if the $G$-action on $X$ has two orbits.
\end{definition}
Suppose $X$ is a two-orbit variety. One orbit is open and the other is closed. The closed orbit is a $K$-orbit. Let $\mup : X \to \liep$ is the gradient map for $G$-action on $X$ and $f_\liep(x) := \frac{1}{2}\parallel\mu_\mathfrak{p}(x)\parallel^2$, $x\in X.$ We recall the following result from \cite{Properties}.
\begin{theorem}\cite{Properties}\label{two-orbit-Morse} If $G$ acts on $X$ with two orbits, then $f_\liep$ is Morse-Bott. It has only two connected critical submanifolds given by the closed $G$-orbit, the stratum associated with the maximum of $f_\liep$ and by a $K$-orbit, the stratum associated with the minimum of $f_\liep$.
\end{theorem}

The set $\{z\in X: f_\liep(z) = max_{x\in X}f_\liep(x)\}$ is a closed $G$-orbit and so, a $K$-orbit and $\{z\in X: f_\liep(z) = min_{x\in X}f_\liep(x)\}$ is a $K$-orbit. So, let $\{z\in X: f(z) = max_{x\in X}f(x)\} = K \cdot p_1$ and $\{z\in X: f_\liep(z) = max_{x\in X}f_\liep(x)\} = K \cdot p_2$, where $p_1, p_2\in X.$

Let $\liea \subset \liep$ be a maximal Abelian subalgebra and $\liea_+$ a positive Weyl chamber. It is well known that $K\cdot \liea_+ = \liep$ and if $K \cdot v = K\cdot w$ for any $v, w\in \liea_+,$ then $v = w,$ see e.g \cite{knapp-beyond}.
\begin{theorem}\label{nab}
    If $G$ acts on $X$ with two orbits, then the set $\mu_\liep(X) \cap \liea_+$ is a convex polytope.
\end{theorem}
\begin{proof}
    Let $P = \mu_\liep(X) \cap \liea_+$ and define a function $f: P \to \R,$ given as $$f(v) = \frac{1}{2}\parallel v \parallel^2, \quad v\in P.$$ $P$ is a finite union of polytopes. Then the set
$$\{v\in P: f(v) = max_{w\in P}f(v)\} = P\cap \mu_\liep(K \cdot p_1) = P\cap K\cdot \mu_\liep(p_1) = Q_1.$$ Similarly, 
    $$\{v\in P: f(v) = min_{w\in P}f(v)\} = P\cap \mu_\liep(K \cdot p_2) = P\cap K\cdot \mu_\liep(p_2) = Q_2.$$
    Suppose, $P$ is not convex, then $f$ will attain its maximum on two or more points. This means that there are more than one disjoint maximal sets of $f$ which is a contradiction. Hence, $P$ is a convex polytope.
    
\end{proof}

\medskip

\printbibliography

\end{document}